\documentclass[preprint,3p,12pt]{elsarticle}

\usepackage{amssymb}
\usepackage{amsmath}
\usepackage{amsthm}

\usepackage{lineno}

\usepackage{tikz}

\journal{Journal of Computational and Applied Mathematics}

\biboptions{sort&compress}
\renewcommand{\vec}[1]{{\ensuremath{\boldsymbol{\mathrm #1}}}}
\newcommand{\ten}[1]{\ensuremath{\boldsymbol{\mathsf{#1}}}}

\newcommand{\vdot}{\boldsymbol{\mathsf{\ensuremath\cdot}}}

\newcommand{\del}{\ensuremath{\nabla}}
\newcommand{\deld}{\ensuremath{\del\vdot}}
\newcommand{\lrp}[1]{\left( #1 \right)}
\newcommand{\LRP}[1]{\bigl( #1 \bigr)}

\newcommand{\FF}{{\mathrm{ff}}}
\newcommand{\PM}{{\mathrm{pm}}}

\newcommand{\M}{\displaystyle M_1^{1,\mathrm{bl}}}
\newcommand{\N}{N_1^{\mathrm{bl}}}
\newcommand{\ex}{\text{ex}}
\newcommand{\e}{\text{e}}

\theoremstyle{theorem}
\newtheorem{theo}{Theorem}[section]

\theoremstyle{definition}

\begin{document}

\begin{frontmatter}



\title{Optimized Schwarz method for the Stokes--Darcy problem with generalized interface conditions}

\author[1]{Paula Strohbeck}
\ead{Paula.Strohbeck@ians.uni-stuttgart.de}

\author[2]{Marco Discacciati}
\ead{M.Discacciati@lboro.ac.uk}

\author[1]{Iryna Rybak\corref{cor1}}
\ead{iryna.rybak@ians.uni-stuttgart.de}

\affiliation[1]{
    organization={Institute of Applied Analysis and Numerical Simulation, University of Stuttgart},   
    addressline={Pfaffenwaldring 57},
    city={Stuttgart},
    postcode={70569},
    country={Germany}}

\affiliation[2]{
    organization={Department of Mathematical Sciences, Loughborough University},
    addressline={Epinal Way},
    city={Loughborough},
    postcode={LE11~3TU},
    country={United Kingdom}}

\cortext[cor1]{Corresponding author}
\begin{abstract}
Due to their wide appearance in environmental settings as well as  industrial and medical applications, the Stokes--Darcy problems with different sets of interface conditions establish an active research area in the community of mathematical modelers and computational scientists. 
For numerical simulation of such coupled problems in applications, robust and efficient computational algorithms are needed. 
In this work, we consider a generalization of the Beavers--Joseph interface condition recently developed using homogenization and boundary layer theory.  This extension is applicable not only for the parallel flows to the fluid--porous interface as its predecessor, but also for arbitrary flow directions. To solve the Stokes--Darcy problem with these generalized interface conditions efficiently, we develop and analyze a Robin--Robin domain decomposition method using Fourier analysis to identify optimal weights in the Robin interface conditions. We study efficiency and robustness of the proposed method and provide numerical simulations which confirm the obtained theoretical results. 
\end{abstract}




\end{frontmatter}


\section{Introduction}
Stokes--Darcy problems with various sets of interface conditions are widely used in the literature to describe fluid flow in coupled systems containing a free-fluid domain in contact with a porous medium. The most famous interface condition is the Beavers--Joseph condition on the tangential velocity component~\cite{Beavers_Joseph_67}. It relates the jump in the tangential velocity to the shear stress across the fluid--porous interface. This condition is often used in the form modified by Saffman~\cite{Saffman_71} and establishes the link between the tangential velocity in the fluid and  the shear stress at the interface, thus neglecting the contribution of the seepage velocity. However, both the Beavers--Joseph and the Beavers--Joseph--Saffman conditions have a limited applicability and are valid only for flows that are parallel or perpendicular to the fluid--porous interface~\cite{Eggenweiler_Rybak_20}. 

There exists several generalizations of the Beavers--Joseph condition, which could be applicable to arbitrary flow directions, e.g.,~\cite{Angot_etal_17, Zampogna_Bottaro_16, Ahmed-bottaro-24, Eggenweiler_Rybak_21, Lacis_Bagheri_17, Carraro_etal_15, Sudhakar_21, Naqvi_Bottaro_2021, Lacis_etal_20}.  
However, some of the coupling strategies are purely theoretical, and include  coefficients which still need to be determined. In this work, we focus on the generalized interface conditions recently developed in~\cite{Eggenweiler_Rybak_21} by means of the homogenization and boundary layer theory. The advantage of these conditions is their applicability for flow systems with arbitrary flow directions to the fluid--porous interface and the ability to compute all the physical parameters appearing in the coupling conditions numerically using information on the pore geometry. These advantages, in comparison to the other coupling conditions available in the literature, are demonstrated, e.g.,~in~\cite{Eggenweiler_Rybak_21, Strohbeck-Eggenweiler-Rybak-23}. The well-posedness of the Stokes--Darcy problem with these generalized interface conditions is proved in our previous work~\cite{Eggenweiler_Rybak_Discacciati_21}. There, the coupled system was studied and solved numerically in the monolithic way. However, for numerical simulation of applications efficient numerical algorithms are of great interest. 

The Stokes--Darcy systems can be decoupled in a natural way at the sharp fluid--porous interface and thus non-overlapping domain decomposition methods can be applied to solve them efficiently. In this case, the original coupled problem is reduced to two smaller separate problems which can be solved independently using appropriate numerical methods in each subdomain, e.g.~\cite{Discacciati_Quarteroni_09, Discacciati_Gerardo-Giorda_18, Chen_etal_2011, He_etal_2015, Liu_etal_2022, Vassilev_etal_2014, Discacciati_etal_2007, Caiazzo_etal_2014, Liu_etal_2021,Quarteroni_Valli_1999}. 
It is well known that the classical Dirichlet--Neumann methods for the Stokes--Darcy problem with the Beavers--Joseph interface condition may suffer from slow convergence in case when the values of fluid viscosity and permeability are small~\cite{Quarteroni_Valli_1999, Discacciati_2004}. A similar behaviour has been observed also with FETI and BDD methods \cite{Galvis:2007:FBP,Galvis:2006:BDD}. In contrast, domain decomposition methods based on Robin--Robin interface conditions have showed better performance as they guarantee a more robust behaviour with respect to the physical parameters. Initial contributions in this direction can be found, e.g., in  \cite{Discacciati_etal_2007,Chen_etal_2011,Cao_etal_2011,Caiazzo_etal_2014,He_etal_2015}, and also \cite{Feng:2012:AMC,Cao:2014:PNI} for the time-dependent Stokes--Darcy problem. However, the key aspect in Robin--Robin methods is the choice of the weighting coefficients in the Robin interface conditions that may lead to poor performance if not carefully selected. Typically, these coefficients are optimized using Fourier analysis in simplified geometrical settings. The resulting iterative methods are referred to as optimized Schwarz methods in the literature (see, e.g., \cite{Gander:2006:OSM}), and they have been successfully applied in various cases (see, e.g., \cite{LGG:SISC:2009,Gander_XU:2016:OSM,gander2019heterogeneous,Gander:2019:SIREV,Chen:2021:OSM,Gigante:2021:OSM,Gigante:2020:OSM}). Robin--Robin methods with optimal coefficients have been studied for the Stokes--Darcy problem with the Beavers--Joseph--Saffman interface condition in the steady and time-dependent cases \cite{Discacciati_Gerardo-Giorda_18,Gander_Vanzan_2020,gander:MOSM,Discacciati:2024:IMAJNA}. 
The objective of this work is to extend the Robin--Robin domain decomposition method proposed in~\cite{Discacciati_Gerardo-Giorda_18} to the Stokes--Darcy problem with the generalized coupling conditions, to determine optimal parameters in the transmission conditions and to analyze the performance of the developed method. 

The manuscript is organized as follows. In section~\ref{sec:model}, the coupled flow model with the generalized interface conditions is formulated. Section~\ref{sec:RR} is devoted to the development and theoretical analysis of the Robin--Robin method. In section~\ref{sec:numerics}, numerical simulation results are provided and the efficiency and robustness of the developed algorithm is studied. Finally, discussion and future work follow in section~\ref{sec:discussion}.

\section{Problem formulation}\label{sec:model}

\subsection{Coupled Stokes--Darcy flow model}

In this work, we consider steady-state incompressible non-inertial flows ($Re \ll 1$) in the free-flow domain $\overline \Omega_\FF \subset \mathbb{R}^2$ and single-fluid-phase flows in the adjacent fully saturated and non-deformable porous medium $\overline \Omega_\PM \subset \mathbb{R}^2$. The whole flow system is assumed to be isothermal. The interface between the two flow domains $\Gamma = \overline \Omega_\FF \cap \overline \Omega_\PM$ is supposed to be flat and does not allow any storage and transport of mass and momentum.

The dimensionless Stokes equations describe fluid flow in the free-flow region
\begin{align}
\deld \vec{v}_\FF = 0, \quad - \deld \ten T(\vec v_\FF,p_\FF) = \vec f_\FF \qquad \text{in} \;\; \Omega_\FF, \label{eq:Stokes}
\end{align}
where $\vec v_\FF$ and $p_\FF$ are the fluid velocity and pressure, respectively, $\ten T(\vec v_\FF,p_\FF) = \nabla \vec v_\FF - p_\FF \ten I$ is the stress tensor,  $\ten I$ is the identity tensor, and $\vec f_\FF$ is the body force.

In the porous-medium domain, the Darcy flow equations 
\begin{align}\label{eq:PM-Darcy}
\deld \vec v_\PM = f_\PM,  \quad  \vec v_\PM = -\ten K \, \del p_\PM \qquad \text{in} \;\; 
\Omega_\PM
\end{align}
are applied, where $\vec v_\PM$ is the seepage velocity, $p_\PM$ is the fluid pressure, $\ten K$ is the permeability tensor, which is symmetric positive definite and bounded, and $f_\PM$ is the source term. In this paper, we consider isotropic ($\ten{K} = \kappa \ten{I}$) and orthotropic porous media ($\ten{K} = \text{diag}(\kappa_{11},\kappa_{22})$) with $\kappa$, $\kappa_{11}$, $\kappa_{22}>0$.

On the external boundary of the free-flow domain $\partial \Omega_\FF \setminus\Gamma$ and the porous-medium domain $ \partial 
\Omega_\PM \setminus\Gamma$, suitable boundary conditions are set to ensure the well-posedness of the problem. They are described in section~\ref{sec:numerics} for the considered examples.

\subsection{Interface conditions}\label{sec:interface-conditions}

For coupling the Stokes--Darcy problem~\eqref{eq:Stokes} and~\eqref{eq:PM-Darcy}, we consider the generalized interface conditions developed in~\cite{Eggenweiler_Rybak_21}, which consist of the conservation of mass across the fluid--porous interface~\eqref{eq:IC-mass}, an extension of the balance of normal forces~\eqref{eq:IC-momentum} and a generalization of the Beavers--Joseph condition~\eqref{eq:IC-tangential}:
\begin{align}\label{eq:IC-mass}
\vec v_\FF \vdot \vec n &= \vec v_\PM \vdot \vec n && \text{on} \;\Gamma,
\\[1ex]
\label{eq:IC-momentum}
-\vec n \vdot \ten T\lrp{\vec v_\FF, p_\FF} \vec n + {N_s^{\mathrm{bl}} \, \vec\tau \vdot \ten T\lrp{\vec v_\FF, p_\FF} \vec n} &= p_\PM
&&  \text{on} \; \Gamma,
\\
\frac{1}{\varepsilon (\vec N^\mathrm{bl} \vdot \vec \tau)} \vec v_\FF \vdot \vec\tau +\vec\tau \vdot \ten T\lrp{\vec v_\FF, p_\FF} \vec n &= {-} \frac{\varepsilon}{\vec N^{\mathrm{bl}} \vdot \vec \tau } ( \ten M^{\mathrm{bl}}\nabla p_\PM )\vdot \vec \tau  &&
 \text{on} \; \Gamma, \label{eq:IC-tangential}
\end{align}
with the unit normal $\vec n$ pointing out from the free-flow domain and the tangential vector $\vec \tau$ on the fluid--porous interface $\Gamma$. The scale separation parameter is $\varepsilon \ll 1$.

The boundary layer coefficients $N_s^{\mathrm{bl}} \in \mathbb{R}$, $\vec N^{\mathrm{bl}} = (N_1^{\mathrm{bl}}, N_2^{\mathrm{bl}})^\top \in \mathbb{R}^2$ and $\displaystyle \ten M^{\mathrm{bl}} = (M_i^{j,\mathrm{bl}})_{i,j=1,2} \in\mathbb{R}^{2 \times 2}$ can be computed using homogenization and boundary layer theory using information on the pore geometry following~\cite{Eggenweiler_Rybak_21}. For isotropic and orthotropic porous media considered in this work, we get $N_s^{\mathrm{bl}} = 0$, thus the second term in equation~\eqref{eq:IC-momentum} disappears.
Note that $\displaystyle \ten M^{\mathrm{bl}}$ can be interpreted as the interfacial permeability tensor~\cite{Strohbeck-Eggenweiler-Rybak-23}.
Moreover, differently from the original formulation in~\cite{Eggenweiler_Rybak_21} and also our previous work~\cite{Eggenweiler_Rybak_Discacciati_21}, the boundary layer constants have opposite signs, i.e., $\N > 0$ and $\M > 0$. This is due to opposite right-hand sides in the boundary layer problems. All the other components of $\vec N^{\mathrm{bl}}$ and $\ten M^{\mathrm{bl}}$ are zero for the horizontal interface $\Gamma$ used in this work.
 
\section{Robin--Robin method}\label{sec:RR}

In this section, we provide the formulation of the Robin--Robin domain decomposition method for the Stokes--Darcy problem with the generalized interface conditions and conduct convergence analysis using the Fourier transform. 

\subsection{Formulation of the Robin--Robin method}

In this section, we derive the Robin--Robin type domain decomposition method for the Stokes--Darcy problem with the generalized interface conditions~\eqref{eq:Stokes}--\eqref{eq:IC-tangential}.
Let $\alpha_\FF > 0$ and \mbox{$\alpha_\PM > 0$} be two parameters. Linear combinations of the interface equations \eqref{eq:IC-mass} and \eqref{eq:IC-momentum} with coefficients $(-\alpha_\FF,1)$ and $(\alpha_\PM,1)$ result in the two Robin interface conditions on $\Gamma$:
\begin{align}
    -\alpha_\FF \vec v_\FF \vdot \vec n - \vec n \vdot \ten T\lrp{\vec v_\FF, p_\FF} \vec n &= -\alpha_\FF \vec v_\PM \vdot \vec n +  p_\PM, \\
    \alpha_\PM \vec v_\PM \vdot \vec n + p_\PM &=\alpha_\PM \vec v_\FF \vdot \vec n - \vec n \vdot \ten T\lrp{\vec v_\FF, p_\FF} \vec n.
\end{align}
Using these conditions, we formulate a Robin--Robin type algorithm where we equivalently rewrite Darcy's flow equations~\eqref{eq:PM-Darcy} as a second-order elliptic problem for the porous-medium pressure $p_\PM$ as follows. Given the initial Darcy pressure $p_\PM^{(0)}$, find the fluid velocity $\vec v^{(m)}_\FF$ and the pressures $p^{(m)}_\FF$ and $p^{(m)}_\PM$ in the free-flow and porous-medium domains
\begin{equation}\label{eq:stokesStep}
\begin{array}{rl}
- \deld \ten T \LRP{\vec v^{(m)}_\FF,p^{(m)}_\FF} = \vec f_\FF, \quad \deld \vec{v}^{(m)}_\FF = 0 & \quad
\text{in} \;\; \Omega_\FF,\\[2mm]
\displaystyle
\frac{1}{\varepsilon (\vec N^\mathrm{bl} \vdot \vec \tau)} \vec v^{(m)}_\FF \vdot \vec\tau {+} \vec\tau \vdot \ten T\LRP{\vec v^{(m)}_\FF, p^{(m)}_\FF} \vec n 
= {-} \frac{\varepsilon}{\vec N^{\mathrm{bl}} \vdot \vec \tau } \LRP{\ten M^{\mathrm{bl}}\nabla p^{(m-1)}_\PM \vdot \vec \tau} & \quad 
 \text{on} \; \Gamma, \\[4.5mm]
 -\alpha_\FF \vec v^{(m)}_\FF \vdot \vec n - \vec n \vdot \ten T\LRP{\vec v^{(m)}_\FF, p^{(m)}_\FF}{\vec n}
 = \; \alpha_\FF \ten K \nabla p^{(m-1)}_\PM \vdot \vec n +  p^{(m-1)}_\PM & \quad 
 \text{on} \; \Gamma,
\end{array}
\end{equation}
and
\begin{equation}\label{eq:darcyStep}
    \begin{array}{rl}
    \displaystyle
       - \nabla \vdot \LRP{ \ten K \nabla p^{(m)}_\PM } = f_\PM  & \quad \text{in} \; \Omega_\PM,  \\[4.5mm]
-\alpha_\PM \ten K \nabla p^{(m)}_\PM \vdot \vec n + p^{(m)}_\PM \; = \; \alpha_\PM \vec v^{(m)}_\FF \vdot \vec n - \vec n \vdot \ten T \LRP{\vec v^{(m)}_\FF, p^{(m)}_\FF} \vec n  & \quad \text{on} \; \Gamma,
    \end{array}
\end{equation}
for the iteration $m \geq 1$ until convergence. In algorithm~\eqref{eq:stokesStep}--\eqref{eq:darcyStep}, suitable boundary conditions are set on the external boundary of the domain $\LRP{\partial \Omega_\FF \cup \partial \Omega_\PM}\backslash \Gamma$.

\subsection{Analysis of the method}

For the analysis, we consider the approach used in~\cite{Discacciati_Gerardo-Giorda_18,Discacciati:2024:IMAJNA} for the Beavers--Joseph--Saffman coupling condition on the fluid--porous interface. However, since the generalized conditions are more complex in comparison to the Beavers--Joseph--Saffman condition, the previous results do not straightforwardly apply to the case of arbitrary flow directions considered in this work. Therefore, further extensions are needed.

We consider a geometrical setting with the flow domains $\Omega_\FF = \{ (x,y) \in \mathbb{R}^2 \, : \, y>0\}$ and $\Omega_\PM = \{ (x,y) \in \mathbb{R}^2 \, : \, y<0\}$ separated by the horizontal interface $\Gamma = \{ (x,y) \in \mathbb{R}^2 \, : \, y=0\}$. The unit normal and tangential vectors at the interface are $\vec n = (0,-1)^\top$ and $\vec \tau = (1,0)^\top$, respectively. Additionally, since we are interested in studying the behavior of the error and all equations are linear, without loss of generality, we can set the source terms $\vec{f}_\FF$ and $f_\PM$ in~\eqref{eq:stokesStep} and~\eqref{eq:darcyStep} equal to zero. Under these assumptions, algorithm~\eqref{eq:stokesStep}--\eqref{eq:darcyStep} can be written as follows
\begin{align}
\label{eq:Stokes-momentum-simplified-1}
        -\LRP{\partial_{xx} v^{(m)}_{1,\FF} + \partial_{yy} v^{(m)}_{1,\FF}}  + \partial_x p^{(m)}_{\FF}  &= 0 && \text{in }\mathbb{R} \times (0, \infty), \\
\label{eq:Stokes-momentum-simplified-2}
         -\LRP{\partial_{xx} v^{(m)}_{2,\FF} + \partial_{yy} v^{(m)}_{2,\FF}}
 +  \partial_y p^{(m)}_{\FF}
  &= 0 && \text{in }\mathbb{R} \times (0, \infty),\\
\label{eq:Stokes-divergence-simplified}
    \partial_x  v^{(m)}_{1,\FF} + \partial_y v^{(m)}_{2,\FF} &= 0 && \text{in }\mathbb{R} \times (0, \infty), \\
\label{eq:generalised-BJ}
    \frac{1}{\varepsilon N_1^{\mathrm{bl}}}  v^{(m)}_{1,\FF} -\partial_y v^{(m)}_{1,\FF} 
    = - \frac{\varepsilon}{N_1^{\mathrm{bl}}} M_1^{1,\mathrm{bl}} &\partial_x p^{(m-1)}_\PM  && \text{on } \mathbb{R} \times \{0\},\\
\label{eq:IC-Stokes}
    \alpha_\FF  v^{(m)}_{2,\FF} - \partial_y v^{(m)}_{2,\FF} + p^{(m)}_{\FF} = -\alpha_\FF \kappa_{22} \partial_y p^{(m-1)}_\PM & + p^{(m-1)}_\PM && \text{on } \mathbb{R} \times \{0\},
\end{align}
and 
\begin{align}
\label{eq:Darcy}
    -\LRP{\kappa_{11} \partial_{xx}  p^{(m)}_\PM + \kappa_{22} \partial_{yy} p^{(m)}_\PM} &= 0 &&\text{in } \mathbb{R} \times (0, \infty), \\[2mm]
\label{eq:IC-Darcy}
    \alpha_\PM  \kappa_{22} \partial_y p^{(m)}_\PM +p^{(m)}_\PM 
    =-\alpha_\PM v^{(m)}_{2,\FF} - &\partial_y v^{(m)}_{2,\FF} + p^{(m)}_{\FF}  &&\text{on }\mathbb{R} \times \{0\}.
\end{align} 

We conduct the convergence analysis in the Fourier space and use the Fourier transform in the direction tangential to the interface $\Gamma$ (which corresponds to the $x$ variable in our simplified geometrical setting):
\[
  \mathcal{F}: w(x,y) \mapsto \widehat{w}(y,k) = \int_{\mathbb{R}} \e^{-ikx} w(x,y) \, dx\,,
\]
where $k$ is the frequency variable.
At the fluid--porous interface $\Gamma$, we define the error reduction factor using the relation
\begin{equation}
\label{eq:reduction-factor}
    \left|\widehat{p}_\PM^{\,(m)}(0,k)\right| = \rho(\alpha_\FF, \alpha_\PM, k) \, \left|\widehat{p}^{\,(m-1)}_\PM(0,k)\right|.
\end{equation}

\begin{theo}\label{thm:3-1}
The error reduction factor $\rho(\alpha_\FF,\alpha_\PM, k)$ of the Robin--Robin algorithm~\eqref{eq:Stokes-momentum-simplified-1}--\eqref{eq:IC-Darcy} is independent of the iteration $m$, and it can be expressed as
\begin{equation}\label{eq:reduction-factor-def}
    \rho(\alpha_\FF, \alpha_\PM, k)= \left|\rho_1(\alpha_\FF,\alpha_\PM, k) - \rho_2(\alpha_\FF,\alpha_\PM, k)\right|,
\end{equation}
    with
    \begin{align}
    \label{eq:rho-1}
    \rho_1(\alpha_\FF,\alpha_\PM, k) &= 
    \frac{\displaystyle \left( 1 - \alpha_\FF \sqrt{\kappa_{11}\kappa_{22}}\, |k|\right) \left( -\alpha_\PM + |k| \, \frac{2 + 3 \varepsilon N_1^{\mathrm{bl}}|k|}{1 + 2 \varepsilon N_1^{\mathrm{bl}}|k|}\right) }{\displaystyle \left( 1 + \alpha_\PM \sqrt{\kappa_{11}\kappa_{22}}\, |k|\right) \left( \phantom{-}\alpha_\FF + |k| \, \frac{2 + 3 \varepsilon N_1^{\mathrm{bl}}|k|}{1 + 2 \varepsilon N_1^{\mathrm{bl}}|k|}\right)}, \\ 
\rho_2(\alpha_\FF,\alpha_\PM, k) &= 
    (\alpha_\FF + \alpha_\PM)
    \frac{\displaystyle M_1^{1,\mathrm{bl}} \frac{\varepsilon^2 k^2}{1+2\varepsilon N_1^{\mathrm{bl}}|k|}}{\displaystyle \left( 1 + \alpha_\PM \sqrt{\kappa_{11}\kappa_{22}}\, |k|\right) \left( \alpha_\FF + |k| \, \frac{2 + 3 \varepsilon N_1^{\mathrm{bl}}|k|}{1 + 2 \varepsilon N_1^{\mathrm{bl}}|k|}\right)}. \label{eq:rho-2}
\end{align}
\end{theo}

\begin{proof} 
Computing the divergence of the momentum equations~\eqref{eq:Stokes-momentum-simplified-1},~\eqref{eq:Stokes-momentum-simplified-2} written in vectorial form, using the incompressibility condition~\eqref{eq:Stokes-divergence-simplified} and multiplying it by $-1$, we get
\[-\LRP{\partial_{xx} p^{(m)}_\FF + \partial_{yy} p^{(m)}_\FF} = 0,\]
yielding the Fourier transform
\begin{equation}
\label{FT-pff}
    -\partial_{yy} \widehat{p}^{\,(m)}_{\FF} + k^2 \widehat{p}^{\,(m)}_\FF = 0 \qquad \text{in } (0, \infty).
\end{equation}
The solution of ODE~\eqref{FT-pff} is
\begin{equation}
    \widehat{p}^{\,(m)}_\FF(y,k) = P^{(m)}(k)\e^{-|k|y} +  Q^{(m)}(k) \e^{|k|y},
\end{equation}
where $P^{(m)}(k)$ and $Q^{(m)}(k)$ are functions dependent on the frequency $k$.
Since the Fourier transform has to be bounded at infinity, we have $Q^{(m)}(k) = 0$ and obtain
\begin{equation}
\label{eq:pFF}
    \widehat{p}^{\,(m)}_\FF(y,k) = P^{(m)}(k) \e^{-|k|y}.
\end{equation}
The function $P^{(m)}(k)$ is uniquely determined using the Fourier transform of the interface condition~\eqref{eq:IC-Stokes}:
\begin{equation}
\label{eq:IC-Stokes-Fourier}
   \alpha_\FF \widehat{v}^{\,(m)}_{2,\FF} - \partial_y \widehat{v}^{\,(m)}_{2,\FF} + \widehat{p}^{\,(m)}_{\FF} = -\alpha_\FF  \kappa_{22} \partial_y \widehat{p}^{\,(m-1)}_\PM + \widehat{p}^{\,(m-1)}_\PM.
\end{equation}
For the porous-medium problem~\eqref{eq:Darcy}, we get
\begin{equation}
\label{eq:FT-pPM}
    \kappa_{11} k^2 \widehat{p}^{\,(m)}_\PM - \kappa_{22} \partial_{yy} \widehat{p}^{\,(m)}_\PM = 0.
\end{equation}
The solution of ODE~\eqref{eq:FT-pPM} is given by
\begin{equation}
\label{eq:pPM}
    \widehat{p}^{\,(m)}_\PM(y,k) = \Phi^{(m)}(k) \e^{\sqrt{\kappa_{11}/\kappa_{22}}|k|y},
\end{equation}
where $\Phi^{(m)}(k)$ is a function of the frequency $k$.  It is uniquely determined by the Fourier transform of the interface condition~\eqref{eq:IC-Darcy}:
\begin{equation}
\label{eq:IC-Darcy-Fourier}
    \alpha_\PM \kappa_{22} \partial_y \widehat{p}^{\,(m)}_\PM +\widehat{p}^{\,(m)}_\PM = -\alpha_\PM \widehat{v}^{\,(m)}_{2,\FF} - \partial_y \widehat{v}^{\,(m)}_{2,\FF} + \widehat{p}^{\,(m)}_{\FF}.
\end{equation}
We compute the normal velocity $\widehat{v}^{\,(m)}_{2,\FF}$ as a function of the pressure $\widehat{p}^{\,(m)}_\FF$ considering the Fourier transform of the momentum balance equation~\eqref{eq:Stokes-momentum-simplified-2}:
\begin{equation}
\label{eq:FT-v2FF}
    k^2 \widehat{v}^{\,(m)}_{2,\FF} -\partial_{yy} \widehat{v}^{\,(m)}_{2,\FF} = |k| P^{(m)}(k) \e^{-|k| y}.
\end{equation}
The solution of ODE~\eqref{eq:FT-v2FF} is
\begin{equation}
\label{eq:v2}
    \widehat{v}^{\,(m)}_{2,\FF} = \left( A^{(m)}(k) + \frac{y}{2}P^{(m)}(k)\right) \e^{-|k|y},
\end{equation}
where $A^{(m)}(k)$ is again a function of the frequency $k$.
Substituting the Fourier transforms of $p_\FF,\, p_\PM$ and $v_{2,\FF}$ given in~\eqref{eq:pFF},~\eqref{eq:pPM} and~\eqref{eq:v2} into the Fourier transforms of the interface conditions~\eqref{eq:IC-Stokes-Fourier} and~\eqref{eq:IC-Darcy-Fourier}, and setting $y=0$, we get
\begin{align}
    (\alpha_\FF + |k|) A^{(m)} (k) + \frac{1}{2}P^{(m)}(k) &= (1-\alpha_\FF  \sqrt{\kappa_{22} \kappa_{11}} |k| ) \Phi^{(m-1)}(k), \label{eq:final-Stokes}\\
    (1+\alpha_\PM \sqrt{\kappa_{22} \kappa_{11}}|k|)\Phi^{(m)}(k) &= (-\alpha_\PM +|k|) A^{(m)}(k) + \frac{1}{2}P^{(m)}(k). \label{eq:final-Darcy}
\end{align}
In order to simplify~\eqref{eq:final-Stokes},~\eqref{eq:final-Darcy} and get rid of $P^{(m)}$, we use the Fourier transform of the momentum equation~\eqref{eq:Stokes-momentum-simplified-1}:
\begin{equation}
\label{eq:FT-v1FF}
    k^2 \widehat{v}^{\,(m)}_{1,\FF}- \partial_{yy} \widehat{v}^{\,(m)}_{1,\FF} = -ikP^{(m)}(k) \e^{-|k|y},
\end{equation}
which has the solution
\begin{equation}
    \widehat{v}^{\,(m)}_{1,\FF}(y,k) = \left(B^{(m)}(k) - \frac{iyk}{2|k|} P^{(m)}(k)\right)\e^{-|k|y}, \label{eq:v1}
\end{equation}
with the function $B^{(m)}(k)$.
Now, we use the Fourier transform of the continuity equation~\eqref{eq:Stokes-divergence-simplified}:
\begin{equation}
    ik \widehat{v}^{\,(m)}_{1,\FF} + \partial_y \widehat{v}^{\,(m)}_{2,\FF} = 0,
\end{equation}
to express $B^{(m)}(k)$ in terms of $A^{(m)}(k)$ and $P^{(m)}(k)$. This yields
\begin{equation}
\label{eq:B}
    B^{(m)}(k) 
    = -i\frac{|k|}{k}A^{(m)}(k) + \frac{i}{2 k}P^{(m)}(k).
\end{equation}
To formulate $P^{(m)}$ in terms of $A^{(m)}$ and $\Phi^{(m-1)}$, we consider the Fourier transform of the interface condition~\eqref{eq:generalised-BJ}:
\begin{equation}
\label{eq:FT-BJ}
    \frac{1}{\varepsilon N_1^{\mathrm{bl}}} \widehat{v}^{\,(m)}_{1,\mathrm{ff}} -  \partial_y \widehat{v}^{\,(m)}_{1,\FF} = -\frac{\varepsilon}{N_1^{\mathrm{bl}}}  \LRP{M_1^{1,\mathrm{bl}} ik \widehat{p}^{\,(m-1)}_\PM}.
\end{equation}
Inserting the Fourier transforms $\widehat{p}^{\,(m-1)}_\PM$ and $\widehat{v}^{\,(m)}_{1,\FF}$ presented in~\eqref{eq:pPM} and~\eqref{eq:v1} into~\eqref{eq:FT-BJ}, using~\eqref{eq:B}, and setting $y=0$, we get
\begin{equation}
    P^{(m)}(k) = A^{(m)}(k) \frac{\displaystyle \frac{1}{\varepsilon N_1^{\mathrm{bl}}|k|}+1}{\displaystyle\frac{1}{\varepsilon N_1^{\mathrm{bl}}2 k^2} + \frac{1}{|k|}} - \Phi^{(m-1)}(k)
    \frac{\displaystyle \frac{\varepsilon}{N_1^{\mathrm{bl}}}M_1^{1,\mathrm{bl}}}{\displaystyle \frac{1}{\varepsilon N_1^{\mathrm{bl}}2 k^2}+ \frac{1}{|k|}}.
\end{equation}
Equivalently, we formulate
\begin{equation}
\label{eq:final-P}
    P^{(m)}(k) = C_1(k,\varepsilon,N_1^{\mathrm{bl}}) \, A^{(m)}(k) - C_2(k,\varepsilon,N_1^{\mathrm{bl}},M_1^{1,\mathrm{bl}}) \, \Phi^{(m-1)}(k),
\end{equation}
where
\begin{equation}
    C_1(k,\varepsilon,N_1^{\mathrm{bl}}) = 2|k| \, \frac{1+\varepsilon N_1^{\mathrm{bl}}|k|}{1+2\,\varepsilon N_1^{\mathrm{bl}}|k|}, \quad
    C_2(k,\varepsilon,N_1^{\mathrm{bl}},M_1^{1,\mathrm{bl}}) = M_1^{1,\mathrm{bl}} \, \frac{2\varepsilon^2k^2}{1+2\,\varepsilon N_1^{\mathrm{bl}}|k|}.
\end{equation}
In conclusion, we obtain the boundary conditions~\eqref{eq:final-Stokes} and~\eqref{eq:final-P} on $\Gamma$ for the Stokes problem and~\eqref{eq:final-Darcy} for the Darcy problem. Now, we substitute $P^{(m)}(k)$ given in~\eqref{eq:final-P} into~\eqref{eq:final-Stokes} and~\eqref{eq:final-Darcy}, and then we substitute $A^{(m)}(k)$ from \eqref{eq:final-Stokes} into \eqref{eq:final-Darcy}. Finally, using algebraic manipulations, we obtain 
\begin{align*}
    \left|\Phi^{(m)}(k)\right| = \rho(\alpha_\FF,\alpha_\PM,k) \,\left|\Phi^{(m-1)}(k)\right|,
\end{align*}
with the reduction factor $\rho(\alpha_\FF,\alpha_\PM,k)$ defined in~\eqref{eq:reduction-factor-def}--\eqref{eq:rho-2}.
\end{proof}

The expression of the reduction factor $\rho$ derived in Theorem~\ref{thm:3-1} is too complex for further analysis of the method. Therefore, we now obtain 
a more manageable reduction factor $\widetilde{\rho}$ under reasonable assumptions. 
Note that $\rho_2(\alpha_\FF,\alpha_\PM,k) > 0$ for all $k\not= 0$ and both $\rho_1(\alpha_\FF,\alpha_\PM,k)$ and $\rho_2(\alpha_\FF,\alpha_\PM,k)$ are symmetric with respect to $k$. Therefore, we can restrict ourselves to the case $k>0$. First, we simplify $\rho_1(\alpha_\FF,\alpha_\PM,k)$ from~\eqref{eq:rho-1} taking into account that $(2 + 3 \varepsilon N_1^{\mathrm{bl}}|k|)/(1 + 2 \varepsilon N_1^{\mathrm{bl}}|k|)\in \left[3/2, 2\right]$. Then, we neglect $\rho_2(\alpha_\FF,\alpha_\PM,k)$ defined in~\eqref{eq:rho-2}, since the scale separation parameter $\varepsilon \ll 1$ and we have there $\varepsilon^2$. With these modifications, we obtain the simplified reduction factor $\widetilde{\rho}$:
\begin{equation}
\label{eq:rho-simplified}
    \widetilde{\rho}\,(\alpha_\FF,\alpha_\PM,k) = 
    \frac{\displaystyle \left( 1 - \alpha_\FF \sqrt{\kappa_{11}\kappa_{22}}\, k\right) \left( -\alpha_\PM + 2k \right) }{\displaystyle \left( 1 + \alpha_\PM \sqrt{\kappa_{11}\kappa_{22}}\, k\right) \left( \phantom{-}\alpha_\FF + 2 k \right)}.
\end{equation}
To accelerate the convergence of the method, we minimize the reduction factor $\widetilde{\rho}\,(\alpha_\FF, \alpha_\PM,k)$ given in~\eqref{eq:rho-simplified} over all relevant frequencies of the problem, $k \in [k_{\min}, k_{\max}]$, using the classical min-max approach
\begin{equation*}
    \min_{\alpha_\FF,\alpha_\PM > 0}
    \max_{k \in [k_{\min},k_{\max}]}
    |\widetilde{\rho}\,(\alpha_\FF,\alpha_\PM,k)|.
\end{equation*}
The exact values that minimize the reduction factor $\widetilde{\rho}\,(\alpha_\FF, \alpha_\PM,k)$ in~\eqref{eq:rho-simplified} are $\alpha_\FF^{\text{ex}}(k) = 1 / (\sqrt{\kappa_{11}\kappa_{22}} k)$ and $\alpha_\PM^{\text{ex}}(k) = 2 k$. Due to their dependency on the frequency $k$, we cannot use them directly. Considering the relation $\alpha_\FF^{\ex}(k)\,\alpha_\PM^{\ex}(k) = 2/\sqrt{\kappa_{11} \kappa_{22}}$, we restrict the search to the curve
\begin{equation}
    \label{eq:search-curve}
    \alpha_{\FF}\alpha_{\PM} = \frac{2}{\sqrt{\kappa_{11}\kappa_{22}}}.
\end{equation}
\begin{theo}
\label{theo:min-max}
    The solution of the min-max problem
    \begin{equation}
    \label{eq:min-max-problem}
         \min_{\alpha_{\FF}\alpha_{\PM} = \frac{2}{\sqrt{\kappa_{11}\kappa_{22}}}} \max_{k\in [k_{\min},k_{\max}]} \widetilde{\rho}\,(\alpha_{\FF},\alpha_{\PM}, k)
    \end{equation}
    is given by the pair
    \begin{align*}
        \alpha_{\FF}^{*} &=  -\frac{2\sqrt{\kappa_{11}\kappa_{22}}k_{\min}k_{\max} - 1}{\sqrt{\kappa_{11}\kappa_{22}}(k_{\min}+k_{\max})} + \sqrt{\left(\frac{2\sqrt{\kappa_{11}\kappa_{22}}k_{\min}k_{\max} - 1}{\sqrt{\kappa_{11}\kappa_{22}}(k_{\min}+k_{\max})}\right)^2 + \frac{2}{\sqrt{\kappa_{11}\kappa_{22}}}},\\
        \alpha_{\PM}^* &= \frac{2\sqrt{\kappa_{11}\kappa_{22}}k_{\min}k_{\max} - 1}{\sqrt{\kappa_{11}\kappa_{22}}(k_{\min}+k_{\max})} + \sqrt{\left(\frac{2\sqrt{\kappa_{11}\kappa_{22}}k_{\min}k_{\max} - 1}{\sqrt{\kappa_{11}\kappa_{22}}(k_{\min}+k_{\max})}\right)^2 + \frac{2}{\sqrt{\kappa_{11}\kappa_{22}}}}.
    \end{align*}
\end{theo}
\begin{proof}
    Under assumption~\eqref{eq:search-curve} the reduction factor~\eqref{eq:rho-simplified} reads
    \begin{equation}
    \label{eq:rho}
        \widetilde{\rho}\,(\alpha_{\FF},k) = \frac{2}{\sqrt{\kappa_{11}\kappa_{22}}} \left(\frac{1-\alpha_{\FF}\sqrt{\kappa_{11}\kappa_{22}}k}{\alpha_{\FF} + 2 k}\right)^2.
    \end{equation}
    To compute the optimal parameter, we follow the same approach as in~\cite[Proposition 3.3]{Discacciati_Gerardo-Giorda_18}. The value $\alpha_{\FF}^* > 0$ minimizing the reduction factor~\eqref{eq:rho} is given by
    $\widetilde{\rho}\,(\alpha_{\FF}^*, k_{\min}) = \widetilde{\rho}\,(\alpha_{\FF}^*, k_{\max}).$
    This is equivalent to solving the algebraic equation
    \begin{equation}
    \label{eq:a-sq}
        \alpha_{\FF}^2+ 2\alpha_{\FF}\frac{2\sqrt{\kappa_{11}\kappa_{22}}k_{\min}k_{\max} - 1}{\sqrt{\kappa_{11}\kappa_{22}}(k_{\min}+k_{\max})}-\frac{2}{\sqrt{\kappa_{11}\kappa_{22}}} = 0.
    \end{equation}
    The solutions are given by
    \begin{align*}
        \alpha_{\FF}^{*} &=  - \frac{2\sqrt{\kappa_{11}\kappa_{22}}k_{\min}k_{\max} - 1}{\sqrt{\kappa_{11}\kappa_{22}}(k_{\min}+k_{\max})} + \sqrt{\left(\frac{2\sqrt{\kappa_{11}\kappa_{22}}k_{\min}k_{\max} - 1}{\sqrt{\kappa_{11}\kappa_{22}}(k_{\min}+k_{\max})}\right)^2 + \frac{2}{\sqrt{\kappa_{11}\kappa_{22}}}},\\
        \alpha_{\PM}^* &= \frac{2\sqrt{\kappa_{11}\kappa_{22}}k_{\min}k_{\max} - 1}{\sqrt{\kappa_{11}\kappa_{22}}(k_{\min}+k_{\max})} + \sqrt{\left( \frac{2\sqrt{\kappa_{11}\kappa_{22}}k_{\min}k_{\max} - 1}{\sqrt{\kappa_{11}\kappa_{22}}(k_{\min}+k_{\max})}\right)^2 + \frac{2}{\sqrt{\kappa_{11}\kappa_{22}}}}.
    \end{align*}
    \end{proof}
    Note that the reduction factor~\eqref{eq:rho} satisfies  $\widetilde{\rho}\,(\alpha_\FF^*,k) < 1$ for all $k\in [k_{\min},k_{\max}]$.

\subsection{Interface system associated with the Robin--Robin method}

For the purpose of the implementation, it is convenient to reformulate the Robin--Robin method \eqref{eq:stokesStep}--\eqref{eq:darcyStep} as a linear system for suitably chosen interface variables. To this aim, for the iteration $m$ of the algorithm, we introduce the auxiliary variables
\begin{align*}
    \lambda_{\Gamma}^{(m-1)} &= - \frac{\varepsilon}{\vec N^{\mathrm{bl}} \vdot \vec \tau } \LRP{\ten M^{\mathrm{bl}}\nabla p^{(m-1)}_\PM \vdot \vec \tau }, \\
    \lambda_\PM^{(m-1)} &= \alpha_\FF \ten K \nabla p^{(m-1)}_\PM \vdot \vec n +  p^{(m-1)}_\PM, \\
    \lambda_\FF^{(m)} &= \alpha_\PM \vec v^{(m)}_\FF \vdot \vec n - \vec n \vdot \ten T (\vec v^{(m)}_\FF, p^{(m)}_\FF) \vec n.
\end{align*}

Using this notation and denoting by $\vec V_\FF$ a suitable subspace of $H^1(\Omega_\FF)$ to account for possible Dirichlet boundary conditions on the Stokes velocity, the weak form of the Stokes problem \eqref{eq:stokesStep} becomes: find $\vec v_\FF^{(m)} \in \vec V_\FF$ and $p_\FF^{(m)} \in L^2(\Omega_\FF)$ such that, for all $\vec u \in \vec V_\FF$ and $q \in L^2(\Omega_\FF)$, 
\begin{equation}\label{eq:stokesWeak}
\begin{array}{l}
\displaystyle
 \int_{\Omega_\FF} \nabla \vec v_\FF^{(m)} \boldsymbol{:} \nabla \vec u - \int_{\Omega_\FF} p_\FF^{(m)} \nabla \vdot \vec u + \int_\Gamma \alpha_\FF \LRP{\vec v_\FF^{(m)} \vdot \vec n}(\vec u \vdot \vec n)   \\[4.5mm]
\displaystyle
+\int_\Gamma \frac{1}{\varepsilon (\vec N^{\mathrm{bl}} \vdot \vec \tau)} \LRP{\vec v_\FF^{(m)} \vdot \vec \tau}(\vec u \vdot \vec \tau) =\int_{\Omega_\FF} \vec f_\FF \vdot \vec u
 - \int_\Gamma \lambda_\PM^{(m-1)} (\vec u \vdot \vec n) + \int_\Gamma \lambda_\Gamma^{(m-1)} (\vec u \vdot \vec \tau), \\[4.5mm]
\displaystyle
 \int_{\Omega_\FF} -q \, \nabla \vdot \vec v_\FF^{(m)} = 0.
\end{array}
\end{equation}

Moreover, letting  $V_\PM$ be a suitable subspace of $H^1(\Omega_\PM)$ to account for possible Dirichlet boundary conditions on the Darcy pressure, the weak formulation of the Darcy problem \eqref{eq:darcyStep} becomes: find $p_\PM^{(m)} \in V_{\PM}$ such that, for all $\psi \in V_{\PM}$,
\begin{equation}\label{eq:darcyWeak}
    \int_{\Omega_{\PM}} \ten K \nabla p_\PM^{(m)} \vdot \nabla \psi + \int_{\Gamma} \frac{1}{\alpha_\PM} p_\PM^{(m)} \, \psi =  \int_{\Omega_\PM} f_\PM \, \psi + \int_\Gamma \frac{1}{\alpha_\PM} \lambda_\FF^{(m)} \, \psi \, .
\end{equation}
Finally, remark that thanks to the last condition in \eqref{eq:stokesStep}, it holds
\begin{equation}\label{eq:updateStep}
\lambda_{\FF}^{(m)} 
= \lambda_\PM^{(m-1)} +(\alpha_\FF + \alpha_\PM) \vec{v}_\FF^{(m)}\vdot \vec n \qquad \text{on }\Gamma.
\end{equation}
For all $m$, we denote
\begin{equation*}
\vec{\lambda}_\PM^{(m)} = \int_{\Gamma} \lambda_{\PM}^{(m)} (\vec{u}\vdot \vec{n}),\,
\quad
\vec{\lambda}_{\Gamma}^{(m)} =\int_\Gamma \lambda_\Gamma^{(m)} (\vec u \vdot \vec \tau)\,,
\quad
\vec{\lambda}_\FF^{(m)} = \int_{\Gamma} \lambda_{\FF}^{(m)} \psi\,.
\end{equation*}

Consider now a Galerkin finite element approximation of the Stokes--Darcy problem on a computational grid that is conforming at the interface $\Gamma$. For simplicity, we assume that inf-sup stable finite elements are used for the Stokes equations and that Lagrangian elements discretize Darcy's pressure $p_\PM$.
Let the subindices $I_{\FF},\, I_{\PM}$ and $\Gamma$ denote the internal degrees of freedom in $\Omega_{\FF},\, \Omega_{\PM}$ and on the interface $\Gamma$, respectively. Then, with obvious choice of notation, the algebraic form of the Stokes problem \eqref{eq:stokesWeak} becomes
\begin{equation}\label{eq:stokesAlgebraic}
\underbrace{
\begin{pmatrix}
A_{I_{\FF},I_{\FF}} & A_{I_{\FF},\Gamma} & B_{I_{\FF}} \\
A_{\Gamma, I_{\FF}} & A_{\Gamma,\Gamma} + \alpha_\FF M_{\Gamma,\Gamma} & B_{\Gamma} \\
B^\top_{I_{\FF}} & B^\top_{\Gamma} & 0 \\
\end{pmatrix}
}_{=A_\FF}
\begin{pmatrix}
\vec{v}_{\FF,I_{\FF}}^{(m)}\\
\vec{v}_{\FF,\Gamma}^{(m)} \\
p_{\FF}^{(m)}
\end{pmatrix} =
\begin{pmatrix}
\vec{f}_{\FF,I_{\FF}} +\vec{\lambda}^{(m-1)}_{\Gamma}\\
\vec{f}_{\FF,\Gamma} - \vec{\lambda}^{(m-1)}_{\PM}\\
\vec 0
\end{pmatrix},
\end{equation}
where $\vec{v}_{\FF,\Gamma}^{(m)}$ denotes the vector of the degrees of freedom of the normal velocity $\vec v_\FF^{(m)} \vdot \vec n$ on~$\Gamma$. On the other hand, the algebraic form of Darcy's problem \eqref{eq:darcyWeak} is
\begin{equation}\label{eq:darcyAlgebraic}
\underbrace{
\begin{pmatrix}
C_{I_{\PM},I_{\PM}} & C_{I_{\PM},\Gamma} \\
C_{\Gamma,I_{\PM}} & C_{\Gamma,\Gamma} + \alpha_\PM^{-1}M_{\Gamma,\Gamma}
\end{pmatrix}
}_{=A_\PM}
\begin{pmatrix}
p_{\PM, I_{\PM}}^{(m)}\\
p_{\PM,\Gamma}^{(m)} 
\end{pmatrix} = 
\begin{pmatrix}
\vec{f}_{\PM,I_{\PM}}\\
\vec{f}_{\PM,\Gamma}+\alpha_\PM^{-1}\,\vec{\lambda}^{(m)}_{\FF}
\end{pmatrix}.
\end{equation}
Algorithm \eqref{eq:stokesStep}--\eqref{eq:darcyStep} can be rewritten in algebraic form as: given $\vec \lambda_\Gamma^{(0)}$ and $\vec \lambda_\PM^{(0)}$, for $m \geq 1$ until convergence,
\begin{enumerate}
\item
Solve the Stokes problem \eqref{eq:stokesAlgebraic}.

\item
Compute
\begin{equation}\label{eq:step2Algebraic}
\vec{\lambda}^{(m)}_{\FF} = \vec{\lambda}_{\PM}^{(m-1)} + (\alpha_\FF+\alpha_\PM)M_{\Gamma,\Gamma} \vec{v}_{\FF,\Gamma}^{(m)}.
\end{equation}

\item
Solve Darcy's problem \eqref{eq:darcyAlgebraic}.

\item
Compute
\begin{eqnarray}
\vec \lambda^{(m)}_{\Gamma} &=& -\frac{\varepsilon}{\vec N^{\mathrm{bl}} \vdot \vec \tau } \LRP{\ten M^{\mathrm{bl}}\,\nabla p^{(m)}_\PM \vdot \vec \tau}, \label{eq:step4aAlgebraic} \\
    \vec{\lambda}_\PM^{(m)} &=& -\frac{\alpha_{\FF}}{\alpha_{\PM}} \vec{\lambda}_{\FF}^{(m)} + \left(\frac{\alpha_{\FF}}{\alpha_{\PM}} + 1\right)M_{\Gamma,\Gamma}\,p_{\PM,\Gamma}^{(m)}\,. \label{eq:step4bAlgebraic}
\end{eqnarray}

\end{enumerate}

We now rewrite steps 1 and 2 as an interface equation for the unknown $\vec{\lambda}_\FF^{(m)}$. To this aim, let $R_{\Omega_\FF\to\Gamma}$ be the restriction operator that associates the Stokes normal velocity on $\Gamma$ to the Stokes velocity and pressure in $\Omega_{\FF}$ :
\[
  R_{\Omega_\FF\to\Gamma} :
     \begin{pmatrix}
      \vec{v}_{\FF,I_\FF} \\
      \vec{v}_{\FF,\Gamma} \\
      p_\FF
     \end{pmatrix}
     \to
     \vec{v}_{\FF,\Gamma}.
\]
Denote $\vec \eta_\PM = (-\vec{\lambda}_{\Gamma}, \, \vec{\lambda}_{\PM})^\top$ and let $\widetilde{R}_{\Gamma \to \Omega_\FF}$ be the extension operator that, given $\vec \eta_\PM$ on $\Gamma$, constructs the vector at the right-hand side of \eqref{eq:stokesAlgebraic}, i.e.,
\begin{equation*}
  \widetilde{R}_{\Gamma\to\Omega_\FF} :
     \vec \eta_\PM
     \to
     \begin{pmatrix}
      -\vec{\lambda}_{\Gamma} \\
       \vec{\lambda}_{\PM} \\
       \mathbf{0}
     \end{pmatrix}.
\end{equation*}
Using these operators, from \eqref{eq:stokesAlgebraic}, we can write
\begin{equation*}
\vec{v}_{\FF,\Gamma}^{(m)} 
= - R_{\Omega_\FF\to\Gamma}
    A_\FF^{-1}
   \widetilde{R}_{\Gamma\to\Omega_\FF} \vec{\eta}^{(m-1)} \\
   +
   R_{\Omega_\FF\to\Gamma}
   A_\FF^{-1}
   \begin{pmatrix}
    \vec{f}_{\FF,I_{\FF}}\\
    \vec{f}_{\FF,\Gamma}\\
    \vec 0
   \end{pmatrix}.
\end{equation*}
Substituting this expression into \eqref{eq:step2Algebraic} and denoting
\begin{equation*}
\vec{b}_{\FF,\Gamma}
= (\alpha_\FF + \alpha_\PM) \, M_{\Gamma,\Gamma} \, R_{\Omega_\FF\to\Gamma}
   A_\FF^{-1}
   \begin{pmatrix}
    \vec{f}_{\FF,I_{\FF}}\\
    \vec{f}_{\FF,\Gamma}\\
    \vec 0
   \end{pmatrix},
\end{equation*}
we obtain
\begin{equation*}
\vec{\lambda}_\FF^{(m)} = \vec{\lambda}_\PM^{(m-1)} - (\alpha_\FF + \alpha_\PM) \, M_{\Gamma,\Gamma} \, R_{\Omega_\FF \to \Gamma} \, A_\FF^{-1} \, \widetilde{R}_{\Gamma \to \Omega_\FF} \, \vec{\eta}^{(m-1)} + \vec{b}_{\FF,\Gamma}\,.
\end{equation*}
Noticing that
\begin{equation*}
\vec{\lambda}_\PM^{(m-1)} 
= \begin{pmatrix} 0, & I_{\Gamma,\Gamma} \end{pmatrix}
  \vec{\eta}_\PM^{(m-1)},
\end{equation*}
and letting
\begin{equation*}
S_\FF = (\alpha_\FF + \alpha_\PM) \, M_{\Gamma,\Gamma} \, R_{\Omega_\FF \to \Gamma} \, A_\FF^{-1} \, \widetilde{R}_{\Gamma \to \Omega_\FF} - \begin{pmatrix} 0, & I_{\Gamma,\Gamma} \end{pmatrix},
\end{equation*}
steps 1 and 2 in the algorithm above can be rewritten as: given $\vec{\eta}_\PM^{(m-1)}$, compute $\vec{\lambda}_\FF^{(m)}$:
\begin{equation}\label{eq:gaussSeidel1}
\vec{\lambda}_\FF^{(m)} = - S_\FF \, \vec{\eta}_\PM^{(m-1)} + \vec{b}_{\FF,\Gamma}.
\end{equation}

We focus now on steps 3 and 4 of the Robin--Robin algorithm. Let $R_{\Omega_\PM \to \Gamma}$ be the restriction operator that to all the degrees of freedom in $\Omega_\PM$ associates those on $\Gamma$, and let its transposed be the corresponding extension operator. From \eqref{eq:darcyAlgebraic}, we find
\begin{equation*}
{p}_{\PM,\Gamma}^{(m)}
=
R_{\Omega_\PM\to\Gamma} \, A_\PM^{-1}
\left( 
\begin{pmatrix}
\vec{f}_{\PM,I_{\PM}} \\ \vec{f}_{\PM,\Gamma}
\end{pmatrix}
+
\begin{pmatrix}
\mathbf{0} \\ \alpha_\PM^{-1} \vec{\lambda}_\FF^{(m)}
\end{pmatrix}
\right).
\end{equation*}
Then, by substituting this expression into \eqref{eq:step4bAlgebraic} and upon defining
\begin{equation*}
\vec{b}_{\PM,\Gamma}
= \left( 1 + \frac{\alpha_\FF}{\alpha_\PM} \right)
  M_{\Gamma,\Gamma}\,
  R_{\Omega_\PM \to \Gamma}\, A_\PM^{-1}
  \begin{pmatrix}
  \vec{f}_{\PM,I_\PM} \\
  \vec{f}_{\PM,\Gamma}
  \end{pmatrix}
\end{equation*}
and
\begin{equation}
S_\PM
=
\frac{\alpha_\FF}{\alpha_\PM} I_{\Gamma,\Gamma} - 
\frac{1}{\alpha_\PM} \left( 1 + \frac{\alpha_\FF}{\alpha_\PM} \right)
M_{\Gamma,\Gamma} \, R_{\Omega_\PM\to\Gamma} A_\PM^{-1} \, R_{\Omega_\PM\to\Gamma}^\top \, ,
\end{equation}
we can conclude that steps 3 and 4 are equivalent to: given $\vec{\lambda}_\FF^{(m)}$, compute
\begin{equation}\label{eq:gaussSeidel2}
\vec{\lambda}_\PM^{(m)} = - S_\PM \, \vec{\lambda}_\FF^{(m)} + \vec{b}_{\PM,\Gamma}\,.
\end{equation}
Finally, since $p_\PM^{(m)}$ depends on $\vec{\lambda}_\FF^{(m)}$, let us rewrite \eqref{eq:step4aAlgebraic} as
\begin{equation}\label{eq:gaussSeidel3}
\vec \lambda^{(m)}_{\Gamma} = -\frac{\varepsilon}{\vec N^{\mathrm{bl}} \vdot \vec \tau } \LRP{\ten M^{\mathrm{bl}}\,\nabla p^{(m)}_\PM(\vec{\lambda}_\FF^{(m)} \vdot \vec \tau} =: S_\PM^\tau \, \vec{\lambda}_\FF^{(m)}.
\end{equation}

Notice that, while \eqref{eq:gaussSeidel1} and \eqref{eq:gaussSeidel2} involve the solution of one Stokes and one Darcy problem, respectively, equation \eqref{eq:gaussSeidel3} only requires post-processing of the porous medium pressure $p_\PM^{(m)}$. We can also rewrite \eqref{eq:gaussSeidel2} and \eqref{eq:gaussSeidel3} in compact form as
\begin{equation*}
\begin{pmatrix}
-\vec{\lambda}_\Gamma^{(m)} \\ \vec{\lambda}_\PM^{(m)}
\end{pmatrix}
+
\begin{pmatrix}
S_\PM^\tau \\ S_\PM
\end{pmatrix}
\vec{\lambda}_\FF^{(m)}
=
\begin{pmatrix}
\vec{0} \\ \vec{b}_{\PM,\Gamma}
\end{pmatrix},
\end{equation*}
or, equivalently, with obvious choice of notation,
\begin{equation}\label{eq:gaussSeidel4}
\vec{\eta}_\PM^{(m)} + \widetilde{S}_\PM \, \vec{\lambda}_\FF^{(m)} = \widetilde{\vec{b}}_{\PM,\Gamma}\,.
\end{equation}

Therefore, we can conclude that one iteration of the Robin--Robin algorithm is equivalent to a Gauss--Seidel step to solve the interface system
\begin{equation}\label{eq:interfaceSystem}
\begin{pmatrix}
I_{\Gamma,\Gamma} & S_\FF \\
\widetilde{S}_\PM & I_{\Gamma,\Gamma}
\end{pmatrix}
\begin{pmatrix}
\vec{\lambda}_\FF \\
\vec{\eta}_\PM
\end{pmatrix}
=
\begin{pmatrix}
\vec{b}_{\FF,\Gamma} \\
\widetilde{\vec{b}}_{\PM,\Gamma}
\end{pmatrix}.
\end{equation}
The matrix of the linear system \eqref{eq:interfaceSystem} is not symmetric and it is indefinite. Therefore, system~\eqref{eq:interfaceSystem} can be solved using an iterative method such as, e.g., GMRES and, each iteration of the method requires to solve independently one Stokes and one Darcy problem.

\section{Numerical simulation results}\label{sec:numerics}

In this section, we study the performance of the developed  Robin--Robin method. First, we investigate the robustness of the algorithm with respect to the mesh size $h$ and fixed physical parameters using an analytical solution (Test 1) from our previous work~\cite{Eggenweiler_Rybak_Discacciati_21}. Then, we fix the mesh size $h$ and consider an example with varying physical parameters (Test 2).
In both numerical tests, we consider a finite element discretization on structured meshes that are conforming at the interface $\Gamma$. The Stokes problem is discretized using the inf-sup stable $\mathbb{Q}_2-\mathbb{Q}_1$ finite elements and the porous-medium pressure $p_\PM$ is approximated using $\mathbb{Q}_2$ finite elements. The gradient of the pressure in the porous medium $\nabla p_\PM$, that is needed to update the quantity $\vec{\lambda}_\Gamma$ in \eqref{eq:step4aAlgebraic} (equivalently, \eqref{eq:gaussSeidel3}), is reconstructed using the gradient post-processing method proposed in \cite{Loula:1995:CMAME}. The interface system \eqref{eq:interfaceSystem} is solved by GMRES (without restart) with tolerance $\text{tol}= 10^{-9}$ for the stopping criterion of the residual, while, at each GMRES iteration, the local Stokes and Darcy problems are solved using a direct method.

\subsection{Test 1 (analytical solution)}
Here, we test the developed Robin--Robin algorithm considering the analytical solution for the Stokes--Darcy problem with the generalized interface conditions~\eqref{eq:Stokes}--\eqref{eq:IC-tangential} from~\cite{Eggenweiler_Rybak_Discacciati_21}. The computational domains are $\Omega_\FF=[0,1]\times [0,0.5]$ and $\Omega_\PM=[0,1]\times [0.5, 1]$ with the interface $\Gamma =[0,1]\times \{0.5\}$. The source terms $\vec{f}_\FF$, $f_\PM$ and Dirichlet boundary conditions are chosen in such a way that the exact solution of the coupled problem is
\begin{equation}
\begin{split}
v_{1,\FF}  &= \sin\left(\frac{\pi x}2\right) \cos\left(\frac{\pi y}2\right)\, , \hspace{8.5ex}
p_\FF = \frac{\sqrt{2}}{2}\cos\left(\frac{\pi x}2\right) \left(\frac{\e^{y-0.5}}{\kappa}  - \frac{\pi}{2}\right) \, ,
\\
v_{2,\FF}  &=  -\cos\left(\frac{\pi x}2\right) \sin\left(\frac{\pi y}2\right) \, ,
 \hspace{5ex}
p_\PM = \frac{\sqrt{2}}{2} \cos\left(\frac{\pi x}2\right)\frac{\e^{y-0.5}}{\kappa}.
\end{split}
\label{eq:exact-solution}
\end{equation}
We consider $\varepsilon = 10^{-1}$ and the permeability value $\kappa=10^{-4}$. The exact solution~\eqref{eq:exact-solution} satisfies the generalized interface conditions~\eqref{eq:IC-mass}--\eqref{eq:IC-tangential} for the following boundary layer coefficients $\N = 1/\pi \approx  0.3183$ and $\M = 2\kappa (1+0.5\varepsilon)/(\pi\varepsilon^2)\approx 0.00668$. These values are within a typical range for many pore geometries.

We solve the coupled problem~\eqref{eq:Stokes}--\eqref{eq:IC-tangential} numerically using four computational meshes with mesh size $h = 2^{-(j+2)}$, $j=1,\ldots,4$. The computed values of the optimized parameters $\alpha_\FF$ and $\alpha_\PM$ are indicated in Table~\ref{tab:results-case-1}, where we also report the number of GMRES iterations for the four meshes. Since the number of iteration steps changes only slightly, we conclude the robustness of the method with respect to the mesh size.

\begin{table}[!ht]
\begin{center}
\begin{tabular}{c | c c | c}
\hline
$h$ & $\alpha_\FF$ & $\alpha_\PM$ & \#~iterations \\
\hline
$2^{-3}$ & $2.58 \times 10^2$ & $7.75 \times 10^1$ & 14 \\
$2^{-4}$ & $1.91 \times 10^2$ & $1.05 \times 10^2$ & 16 \\
$2^{-5}$ & $1.61 \times 10^2$ & $1.24 \times 10^2$ & 17 \\
$2^{-6}$ & $1.48 \times 10^3$ & $1.35 \times 10^2$ & 18 \\
\hline
\end{tabular}
\caption{Optimal parameters $\alpha_\FF$ and $\alpha_\PM$ and number of GMRES iterations for different meshes}
\label{tab:results-case-1}
\end{center}
\end{table}

\subsection{Test 2 (general filtration problem)}
Now, we consider the general filtration problem defined in our previous work~\cite{Strohbeck-Eggenweiler-Rybak-23}. Here, we have an arbitrary flow to the fluid--porous interface (see microscale velocity field in Fig.~\ref{fig:general-filtration-problem}) for which the generalized interface conditions~\eqref{eq:IC-mass}--\eqref{eq:IC-tangential} are suitable. The free-flow region is $\Omega_\FF=[0,1]\times [0,0.5]$, the porous medium is $\Omega_\PM=[0,1]\times [-0.5,0]$ so that the interface is $\Gamma=[0,1]\times \{0\}$. We define $\Gamma_{\text{out}} = \left(\{0\} \times [0,0.5] \right) \cup \left( \{1\} \times [0,0.225] \right)$, consider zero source terms in both domains, $\vec{f}_\FF = \vec{0}$ and $f_\PM = 0$, and set the following boundary conditions
\begin{align*}
    \vec{v}_\FF &=  (0,-0.7\operatorname{sin}(\pi x)) &&\text{\; on \; }[0,1] \times \{0.5\} \, ,
    \qquad \
    \vec{v}_\FF = \vec 0 \text{\; on \; }\{1\} \times [0.225,0.5] \, ,
    \\
   v_{1,\FF} &= 0 , \quad \partial v_{2,\FF}/\partial x = 0
   &&\text{\; on \;} \Gamma_{\text{out}},\\
   v_{2,\PM} &= 0 &&\text{\; on \; }\partial \Omega_\PM \backslash \Gamma.
\end{align*}

\begin{figure}[!ht]
\begin{tikzpicture}
    \node at (0,0) {\includegraphics[scale = 1.3]{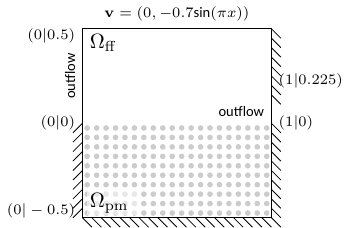}};
    \node at (8,-0.56) {\includegraphics[scale=0.16]{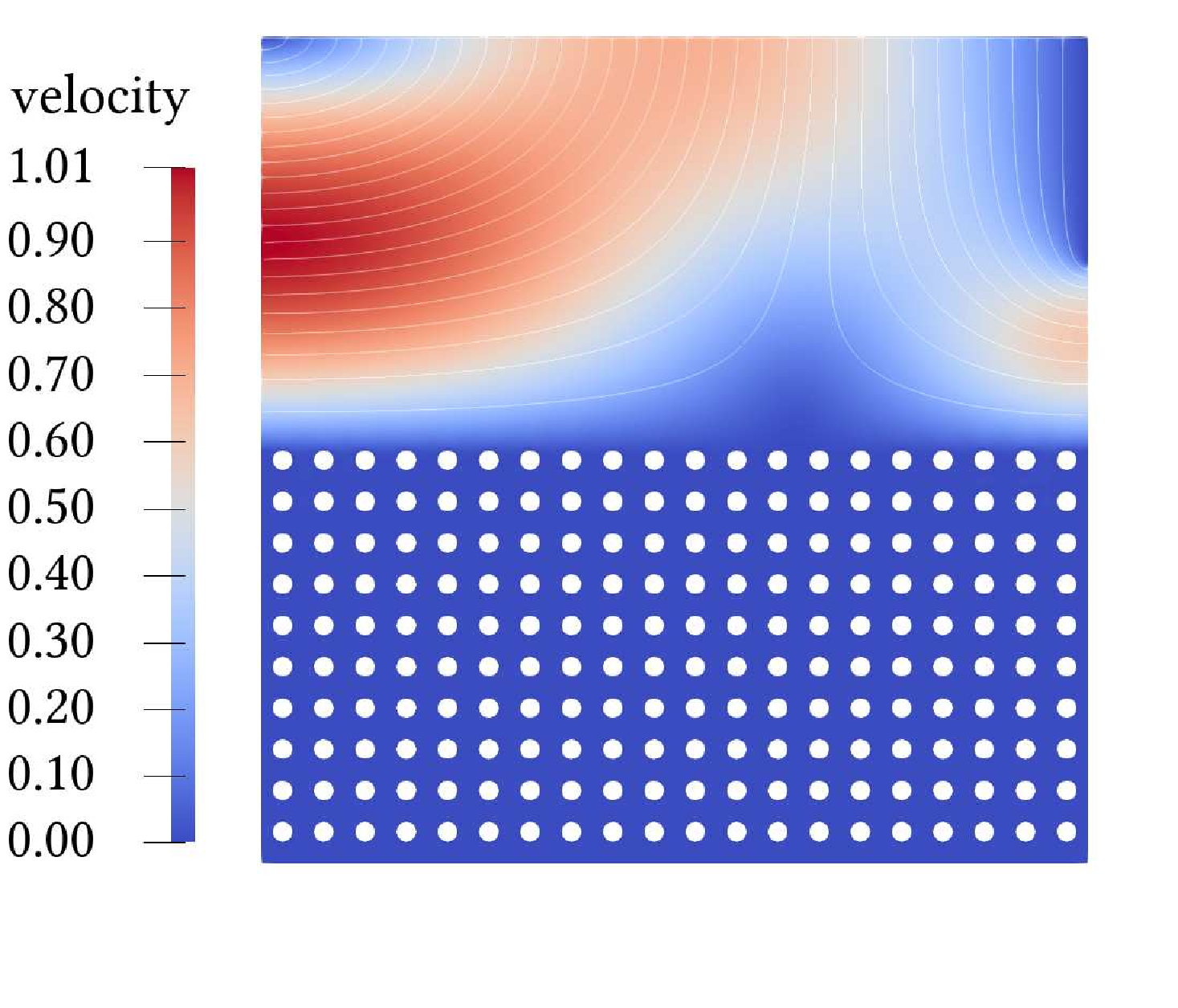}};
\end{tikzpicture}
\caption{Flow system description (left) and visualization~\cite{Strohbeck-Eggenweiler-Rybak-23} of the microscale velocity field (right) for the general filtration problem}
\label{fig:general-filtration-problem}
\end{figure}

We investigate the influence of physical parameters on the convergence rate of the developed Robin--Robin method. The orders of the boundary layer coefficients $\N$ and $\M$ are taken based on our study of boundary layer constants for different pore geometries. We set $\N = 10^{-2}$ since it reflects the order of typical values and vary $\M$, the permeability $\kappa$ and the scale separation parameter $\varepsilon$.
The computational mesh is characterized by $h=0.0125$. The optimal coefficients $\alpha_\FF$ and $\alpha_\PM$ computed for various combinations of the physical parameters are reported in Table~\ref{tab:results-case-2b} together with the number of iterations needed for the convergence of the method. The method shows high robustness with respect to the scale separation parameter $\varepsilon$ and the boundary layer constant $\M$ (Table~\ref{tab:results-case-2b}). However, for the intrinsic permeability $\kappa$, we observe moderate robustness with decreasing number of iteration steps for smaller permeability values. The intrinsic permeability $\kappa$ is indeed the parameter that affects the convergence rate of the algorithm in the most significant way. This can be seen by plotting the error reduction  factor \eqref{eq:reduction-factor-def} and its simplified form \eqref{eq:rho-simplified} versus $k$ as done in Fig.~\ref{fig:convergenceFactorsTest2a} for the combinations of parameters reported in cases 2, 4 and 8 in Table~\ref{tab:results-case-2b}. Notice that $k_{\min} = \pi$ and $k_{\max} = \pi/(h/2)$, where $h$ is the mesh size and the factor 2 accounts for the fact that quadratic elements $\mathbb{Q}_2$ are used to approximate the pressure $p_\PM$. From the graphs (Fig.~\ref{fig:convergenceFactorsTest2a}), first of all we notice that the simplified reduction factor $\widetilde{\rho}$ provides a good approximation of the original reduction factor $\rho$ since the contribution of the term $\rho_2$ is negligible compared to $\rho_1$. Moreover, we notice that in case 2 with $\kappa=10^{-3}$, there is a significant number of error frequencies for which the value of the error reduction factor is above 0.5. This does not occur in the other cases, especially for $\kappa=10^{-7}$, where the error reduction factor is one order of magnitude smaller than in the other two cases. This explains why the number of iterations decreases significantly for smaller values of $\kappa$.

\begin{table}[!ht]
    \centering
    \begin{tabular}{c | c c c |c c | c}
    \hline
    Case & $\kappa$ &$\varepsilon$ & $M_1^{1,\mathrm{bl}}$ & $\alpha_\FF $ & $\alpha_\PM$ & \#~iterations \\
\hline
    1 &  $10^{-2}$ & $10^{-2}$ & $10^{-4}$ & $9.33 \times 10^0$    & $2.14\times 10^1$ & 19 \\
    2 & $10^{-3}$ & $10^{-2}$ & $10^{-4}$ & $4.07\times 10^1$ & $4.92\times 10^1$ & 20 \\
    3 &  $10^{-5}$ & $10^{-2}$ & $10^{-4}$ & $6.78\times 10^2$ & $2.95\times 10^2$ & 16 \\
    4 &  $10^{-7}$ & $10^{-2}$ & $10^{-4}$ & $4.00\times 10^4$ & $5.00\times 10^2$ &  8 \\
     \hline
    5 &  $10^{-5}$ & $10^{-1}$ & $10^{-4}$ & $6.78\times 10^2$ & $2.95\times 10^2$ & 16 \\
    6 &  $10^{-5}$ & $10^{-2}$ & $10^{-4}$ & $6.78\times 10^2$ & $2.95\times 10^2$ & 16 \\
    7 &  $10^{-5}$ & $10^{-3}$ & $10^{-4}$ & $6.78\times 10^2$ & $2.95\times 10^2$ & 16 \\
     \hline
    8 & $10^{-5}$ & $10^{-2}$ & $10^{-3}$ & $6.78\times 10^2$ & $2.95\times 10^2$ & 16 \\
    9 & $10^{-5}$ & $10^{-2}$ & $10^{-4}$ & $6.78\times 10^2$ & $2.95\times 10^2$ & 16 \\
    \hline
    \end{tabular}
\caption{Optimal parameters $\alpha_\FF$ and $\alpha_\PM$ and number of iteration steps for different values of parameters $\kappa,\, \varepsilon$ and $M_1^{1,\mathrm{bl}}$ with $h=0.0125$ and $N_1^{\mathrm{bl}} = 10^{-2}$}
\label{tab:results-case-2b}
\end{table}

\begin{figure}
    \includegraphics[width=0.32\textwidth]{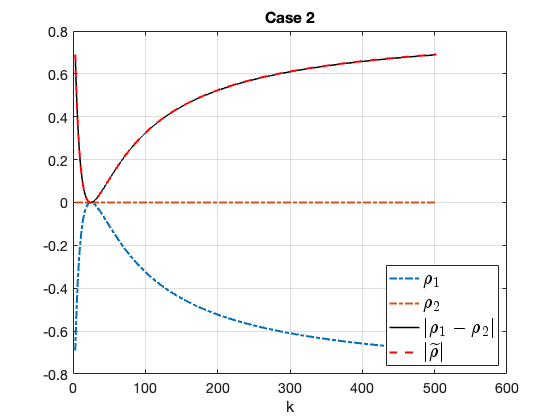}
    \includegraphics[width=0.32\textwidth]{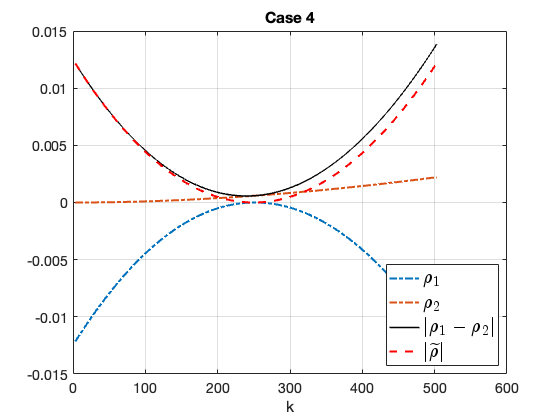}
    \includegraphics[width=0.32\textwidth]{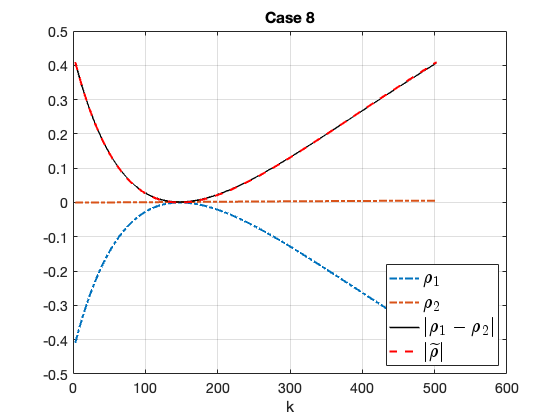}
\caption{Error reduction factors versus relevant frequencies $k$ for the values of the coefficients corresponding to Case 2 (left), Case 4 (middle) and Case 8 (right) in Table~\ref{tab:results-case-2b}}
\label{fig:convergenceFactorsTest2a}
\end{figure}

\section{Discussion}\label{sec:discussion}

In this work, we develop and analyze an optimized Schwarz method for the steady-state Stokes--Darcy problem with generalized interface conditions. These coupling conditions have been recently developed using homogenization and boundary layer theory and are applicable for flows with arbitrary direction at the fluid--porous interface. The work extends the previous results~\cite{Discacciati_Gerardo-Giorda_18}, that were valid only for parallel flows to the interface, to coupled flow systems with general flow directions.

We conduct the convergence analysis in the Fourier space and compute optimal Robin parameters. We study the performance of the developed method with respect to the mesh size and with respect to the physical parameters appearing in the model and in the generalized interface conditions. For this purpose, we consider two different test cases: one with the analytical solution used in our previous work on well-posedness of the coupled model, and one where the flow has arbitrary direction at the fluid--porous interface. The developed method is highly robust with respect to the mesh size, boundary layer coefficients and scale separation parameter appearing in the generalized coupling conditions. The method demonstrates a moderate robustness with respect to the intrinsic permeability such that we get less iterations for the smaller permeability values. This is due to the fact that in such situations the error reduction factor of the Robin--Robin method is much smaller than for higher permeability values.

\section*{Acknowledgement}
The work is funded by the Deutsche Forschungsgemeinschaft (DFG, German Research Foundation) – Project Number 327154368 – SFB 1313 and by the EPSRC grant EP/V027603/1.




\bibliographystyle{elsarticle-harv}
\bibliography{bibliography}



\end{document}